\documentclass[11pt,reqno]{amsart}
\usepackage{amsmath,amsthm,amssymb,enumerate,color,fancyvrb}
\usepackage[a4paper]{geometry}
\usepackage[colorlinks=true,urlcolor=blue,pagebackref=true]{hyperref}
\usepackage{multicol}
\usepackage{setspace}
\theoremstyle{plain}
\newtheorem{theorem}{Theorem}[section]

\newtheorem{lemma}[theorem]{Lemma}
\newtheorem{corollary}[theorem]{Corollary}

\theoremstyle{definition}
\newtheorem{definition}[theorem]{Definition}
\newtheorem{notation}[theorem]{Notation}

\newtheorem{remark}[theorem]{Remark}

\newcommand{\F}{\mathbb{F}}
\newcommand\mb{\mathbf}      
\newcommand{\vp}{\varphi}
\newcommand{\czf}{{\sf CZF}}
\newcommand{\izf}{{\sf IZF}}
\newcommand{\ac}{{\sf AC}}
\newcommand{\cac}{{\sf CAC}}
\newcommand{\dc}{{\sf DC}}
\newcommand\ft{{\sf FT}}
\newcommand{\markov}{{\sf MP}}
\newcommand{\ip}{{\sf IP}}
\newcommand{\rdc}{{\sf RDC}}
\newcommand{\pax}{{\sf PAx}}

\newcommand{\fo}{\Vdash}   
\newcommand{\ap}{\mb p}  
\newcommand{\pair}[1]{\langle{#1}\rangle}
\newcommand{\va}{\mathrm{V_{ex}}(A)}
\newcommand{\vset}[1]{\{{#1}\}_{A}}
\newcommand{\vpair}[1]{\langle{#1}\rangle_{\!A}}

\DeclareMathOperator{\ps}{\mathcal{P}}
\DeclareMathOperator{\succe}{\mb{succ}}
\DeclareMathOperator{\pred}{\mb{pred}}
\newcommand{\imp}{\rightarrow}

\newcommand{\biimp}{\leftrightarrow}

\newcommand{\tuep}[2]{ {#1}^{#2}}
\newcommand{\McCarty}{generic\  }
\newcommand{\GMcCarty}{Generic\  }

\title[Extensional realizability for set theory]{Extensional realizability for  intuitionistic set theory}
\thanks{The authors' research was supported by a grant from John Templeton Foundation \lq\lq A new dawn of Intuitionism: Mathematical and Philosophical advances\rq\rq\ (grant ID 60842).}

\author[Frittaion]{Emanuele Frittaion}
\address{School of Mathematics, University of Leeds, UK}
\author[Rathjen]{Michael Rathjen}
\address{School of Mathematics, University of Leeds, UK}

\begin{document}
	
	
\begin{abstract} In generic realizability for set theories, realizers treat unbounded quantifiers generically. To this form of realizability, we add another layer of extensionality 
by requiring that realizers ought to act extensionally on realizers, giving rise to a realizability universe $\va$ in which the axiom of choice in all finite types, $\ac_{\ft}$,  is  realized, where $A$ stands for an arbitrary  partial combinatory algebra. This construction furnishes ``inner models'' of many set theories that additionally validate $\ac_{\ft}$, in particular
it provides a self-validating semantics for $\czf$ (Constructive Zermelo-Fraenkel set theory) and $\izf$ (Intuitionistic Zermelo-Fraenkel set theory).  
One can also add large set axioms and many other principles.
\end{abstract}
	
\maketitle
\tableofcontents

\section{Introduction}
In this paper we define an extensional version of  generic\footnote{The descriptive attribute ``generic'' for this kind of realizability is due to  McCarty \cite[p.\ 31]{M84}.}
realizability over any given partial combinatory algebra (pca) and prove that it  provides a self-validating semantics for $\czf$ (Constructive Zermelo-Fraenkel set theory) as well as $\izf$ (Intuitionistic Zermelo-Fraenkel set theory), i.e., every theorem of $\czf$ ($\izf$) is realized by just assuming the axioms of $\czf$ ($\izf$)  in the background theory. Moreover, it is shown that the axiom of choice in all finite types, $\ac_{\ft}$, also holds under this interpretation.\footnote{As a byproduct,  we reobtain the already known result (e.g. \cite[4.31, 4.33]{R03m}) that augmenting $\czf$   by $\ac_{\ft}$ does not  increase the stock of provably recursive functions. Likewise, we reobtain the result (a consequence of  \cite{friedman73}) that augmenting $\izf$   by  $\ac_{\ft}$ does not  increase the stock of provably recursive functions.} 
This uniform tool of realizability can be combined  with forcing to show that $\izf+\ac_{\ft}$ is conservative over $\izf$ with respect to arithmetic formulae (and similar results with large set axioms).  For special cases, namely, finite type dependent choice, $\dc_{\ft}$, and finite type countable choice, $\cac_{\ft}$,\footnote{$\dc_{\ft}$ is the scheme $\forall x^{\sigma}\,\exists y^{\sigma}\,\vp(x,y)\to \forall x^{\sigma}\, \exists f^{0\sigma}\,[ f(0)=x\;\wedge\;
\forall n\,\vp(f(n),f(n+1))]$, while  $\cac_{\ft}$ stands for the scheme $\forall n\, \exists y^{\tau}\, \vp(n,y)\to \exists f^{0\tau}\, \forall n\,\vp(n,f(n))$.}  this has been shown  in  \cite[Theorem 5.1]{friedman_scedrov84} and \cite[XV.2]{B85}, but not for  $\ac_{\ft}$. 
The same technology works for $\czf$. However, for several subtheories of $\czf$ (with exponentiation in lieu of subset collection) such  conservativity results  have already  been obtained by Gordeev \cite{gordeev} by very different methods, using total combinatory algebras and an abstract form of realizability combined with genuine 
proof-theoretic machinery. 

\GMcCarty realizability is
markedly  different from Kleene's number and function realizability as well as modified realizability.  It  originates with Kreisel's and Troelstra's \cite{KTr70} definition of realizability
for second order Heyting arithmetic and the theory of species.
Here, the clauses for the realizability relation $\Vdash$ relating
to second order quantifiers are the following: $e\Vdash \forall
X\, \phi(X)\Leftrightarrow \forall X\,e\Vdash \phi(X)$, $e\Vdash
\exists X\, \phi(X)\Leftrightarrow \exists X\,e\Vdash \phi(X)$. This
type of realizability does not seem to give any constructive
interpretation to set quantifiers; realizing numbers ``pass
through" quantifiers. However, one could also say that thereby the
collection of sets of natural numbers is generically conceived.
Kreisel-Troelstra realizability was applied to systems of higher
order arithmetic  and set theory by Friedman \cite{F73} and to further set theories by Beeson \cite{B79}. 
An immediate descendant of the interpretations of Friedman and Beeson was used by McCarty \cite{M84,M86}, who, unlike the realizabilities of Beeson, devised realizability directly for extensional $\izf$: {\em ``we found it a nuisance to interpret the extensional theory into the intensional before realizing." } (\cite[p.\ 82]{M84}).  A further generalization, inspired by a remark of Feferman in \cite{F75},  that McCarty introduced was that he used realizers from applicative structures, i.e. arbitrary models  of Feferman's theory $\mathrm{APP}$, rather than just natural numbers.

\GMcCarty realizability \cite{M84} is based on the construction of a realizability universe $\mathrm{V}(A)$ on top of an applicative structure or partial combinatory algebra $A$.  Whereas in \cite{M84,M86}
the approach is geared towards $\izf$, making  use of transfinite iterations of 
the powerset operation,
it was shown in \cite{R06}  that $\czf$ suffices for a formalization of  $\mathrm{V}(A)$ and the generic realizability based upon it.  This tool has been successfully applied to the proof-theoretic analysis  of  $\czf$ ever since \cite{R05a, R05, R08, S14}. 

With regard to $\ac_{\ft}$, it is perhaps worth mentioning that, by using   \McCarty realizability \cite{M84},   one can show that  $\ac_{0,\tau}$ for $\tau\in\{0,1\}$ holds in the realizability universe $\mathrm{V}(A)$ for  any pca $A$ (cf.\ also \cite{DR19}). For instance, one can take Kleene's first algebra.  With some effort, one can also see that $\ac_{1,\tau}$ for $\tau\in\{0,1\}$ holds in $\mathrm{V}(A)$ by taking, e.g.,  Kleene's second algebra. It is conceivable that one can construct a  specific pca $A$ so as to validate $\ac_{\ft}$ in $\mathrm{V}(A)$. In this paper we show that,  by building extensionality into the realizability universe and by  adapting the definition of realizability, it is possible to satisfy choice in all finite types at once, regardless  of the partial  combinatory algebra $A$ one starts with.

Extensional variants of realizability in the context of (finite type) arithmetic have been investigated by Troelstra (see \cite{T98}) and  van Oosten  \cite{ Oosten97,Oosten08}, as well as \cite{F19, BS18}, and for both arithmetic and set theory by Gordeev in \cite{gordeev}. For earlier references on extensional realizability, in particular \cite{Grayson}, where the notion for first order arithmetic  first appeared,  and \cite{Pitts},  see Troelstra \cite[p.\ 441]{T98}.

\section{Partial combinatory algebras}
Combinatory algebras are the brainchild of Sch\"onfinkel \cite{Schoen24} who presented his ideas in G\"ottingen in 1920. The quest for an optimization of his framework, singling out  
a minimal set of axioms,  engendered much work and writings from 1929 onwards, notably by Curry \cite{Curry29,Curry30}, under the heading of {\em combinatory logic}. 
Curiously, a very natural generalization of Sch\"onfinkel's structures, where the application operation is not required to be always defined, was  axiomatically characterized  only  in 1975 by Feferman in the shape of the most basic axioms of his theory $T_0$ of explicit mathematics \cite{F75}\footnote{In the literature, this subtheory of $T_0$ has been christened $\mathrm{EON}$ (for {\em elementary theory of operations and numbers}; see \cite[p.\ 102]{B85}) and  $\mathrm{APP}$ (on account of comprising the  {\em applicative axioms} of $T_0$; see \cite[Chapter 9, Section 5]{TvD88}). However, to be precise let us point out that $T_0$ as formulated in  \cite{F79} differs from the original formulation in \cite{F75}:
\cite{F79} has a  primitive classification constant $\mathbb N$ for the natural numbers as well as constants for successor and predecessor on $\mathbb N$, and more crucially,  equality is not assumed to be decidable and the definition-by-cases operation is restricted to $\mathbb N$.}
and in \cite[p.\ 70]{F78a}. Feferman called these structures {\em applicative structures}. 

\begin{notation} In order to introduce the notion of a pca, we shall start with that of a partial operational structure $(M,\cdot)$, 
 where $\cdot$ is just a partial binary operation on $M$. We use
$a\cdot b\simeq c$ to convey that $a\cdot b$ is defined and equal to $c$. $a\cdot b\downarrow $ stands for $\exists c\, (a\cdot b\simeq c)$. 
In what follows, instead of $a\cdot b$ we will just write $ab$. We also employ the association to the left convention, meaning that e.g. 
 $abc\simeq d$  stands for the following: there exists $e$ such that $ab\simeq e$ and $ec\simeq d$.
\end{notation}

\begin{definition}
A {\em partial  combinatory algebra}  (pca) is a partial operational structure $(A,\cdot)$  such that $A$ has at least two elements and there are elements 
$\mb k$ and $\mb s$ in $A$ such that $\mb k a$, $\mb sa$ and  $\mb sab$ are always defined, and 
\begin{itemize}
	\item $\mb ka b\simeq a$;
	\item $\mb sabc\simeq ac(bc)$.
\end{itemize}
The combinators $k$ and $s$ are due   to Sch\"onfinkel \cite{Schoen24}   while the axiomatic treatment, although formulated just in the total case,  is due to Curry \cite{Curry30}. The word ``combinatory" appears because of a property known as {\em combinatory completeness} described  next. For more information on pcas see \cite{F75,F79, B85,Oosten08}.
\end{definition}

\begin{definition}
Given a pca $A$, one can form application terms over $A$ by decreeing that: 
\begin{enumerate}[(i)]
	\item  variables $x_1,x_2,\ldots$ and  the constants $\mb k$ and $\mb s$ are applications terms over $A$;  
	\item   elements of $A$ are application terms over $A$; 
	\item given application terms $s$ and $t$ over $A$, $(ts)$ is also an application term over $A$.
\end{enumerate}
 Application terms over $A$ will also be called $A$-terms. Terms generated solely by clauses (i)--(iii), will be called {\em application terms}. 

An $A$-term $q$ without free variables has an obvious interpretation $q^A$ in $A$ by interpreting elements of $A$ by themselves and letting $(ts)^A$ be $t^A\cdot s^A$ with $\cdot$ being
the partial operation of $A$. Of course, $q$ may fail to denote an element of $A$. We write $A\models q\downarrow$ (or just $ q\downarrow$) if it does, i.e., if $q^A$ yields an element of $A$.
\end{definition}

The combinatory completeness of a pca $A$   is encapsulated in  $\lambda$-abstraction (see \cite[p.\ 95]{F75}, \cite[p.\ 63]{F79}, and \cite[p.\ 101]{B85} for more details).

\begin{lemma}[$\lambda$-abstraction]
For every term $t$ with variables among the distinct variables $x,x_1,\ldots,x_n$, one can find in an effective way a new term $s$, denoted $\lambda x.t$, such that  
\begin{itemize}
	\item the variables of $s$ are the variables of $t$ except for $x$, 
	\item $s[a_1/x_1,\ldots,a_n/x_n]\downarrow$ for all $a_1,\ldots,a_n\in A$,
	\item  $(s[a_1/x_1,\ldots,a_n/x_n]) a\simeq t[a/x,a_1/x_1,\ldots,a_n/x_n]$ for all $a,a_1,\ldots,a_n\in A$.
\end{itemize}

The term $\lambda x.t$ is built solely with the aid of $\mb k, \mb s$ and symbols occurring in $t$.
\end{lemma}

An immediate consequence of the foregoing abstraction lemma   is the recursion theorem for pca's (see \cite[p.\ 96]{F75}, \cite[p.\ 63]{F79}, \cite[p.\ 103]{B85}).  
\begin{lemma}[Recursion theorem] There exists a closed application term $\mb f$ such that for  every pca $A$ and $a,b\in A$ we have $A\models {\mb f}\downarrow$ 
	and
\begin{itemize}
	\item $A\models \mb f a\downarrow$; 
	\item $A\models \mb f ab\simeq a(\mb fa) b$. 
\end{itemize}
\end{lemma}

\begin{proof}
The heuristic approach consists in finding a fixed point of the form $cc$. Let us search for
$\mb f$ satisfying $\mb fa\simeq cc$, and hence find  a solution of the equation
	\[       ccb\simeq a(cc)b. \]
By using $\lambda$-abstraction, we can easily arrange to have, for every $d$,
	\[       cdb\simeq a(dd)b. \]
Indeed, let $\mb f:=\lambda a.cc$, where $c:=\lambda db.a(dd)b$. Then $f$ is as desired. 
\end{proof}

In every pca, one has pairing and unpairing\footnote{Let $\mb p=\lambda xyz.zxy$, $\mb{p_0}:=\lambda x.x\mb k$, and $\mb{p_1}:=\lambda x.x\bar{\mb k}$, where $\bar{\mb k}:=\lambda xy.y$. Projections $\mb{p_0}$ and $\mb{p_1}$ need not be total. For realizability purposes, however, it is not necessary to have total projections.} combinators $\mb p$,  $\mb {p_0}$, and  $\mb {p_1}$ such that: 
\begin{itemize}
	\item  $\mb pab\downarrow$;
	\item   $\mb {p_i}(\mb pa_0a_1)\simeq a_i$.
\end{itemize}

Generic realizability is based on partial combinatory algebras with some additional structure (see  however Remark \ref{remark}).

\begin{definition}
We say that $A$ is a pca over $\omega$ if there are extra combinators $\mb{succ}, \mb{pred}$ (successor and predecessor combinators), $\mb d$ (definition by cases combinator), and a  map
$n\mapsto \bar n$ from $\omega$ to $A$ such that for all $n\in \omega$ 
	\begin{align*}
	\succe \bar n&\simeq \overline{n+1}, & \pred\overline{n+1}&\simeq \bar n,
	\end{align*}
	\[ \mb d\bar n\bar mab\simeq
	\begin{cases} a & n=m;\\ b & n\neq m. \end{cases}\]
One then defines $\mb 0:=\bar 0$ and $\mb 1:=\bar 1$. 
	
The notion of a pca over $\omega$  coincides with  the notion of $\omega$-pca$^+$  in, e.g., \cite{R05a}.
\end{definition}

Note that one can do without $\mb k$ by letting $\mb k:=\mb d\mb 0\mb 0$. The existence of $\mb d$ implies that the map $n\mapsto \bar n$ is one-to-one. In fact, suppose $\bar n=\bar m$ but $n\neq m$. Then $\mb d\bar n\bar n\simeq \mb d\bar n\bar m$. It then follows that $a\simeq \mb d\bar n\bar nab\simeq \mb d\bar n\bar mab\simeq b$ for all $a,b$. On the other hand, by our definition, every pca contains at least two elements.

\begin{remark}\label{remark}
The notion of a pca over $\omega$ is slightly impoverished one compared to that of a model of  Beeson's theory $\mathbf{PCA}^+$ 
\cite[VI.2]{B85} or Feferman's applicative structures \cite{F79}. However, for our purposes all  the differences between these structures are immaterial as every pca can be expanded to a model of $\mathbf{PCA}^+$, which at the same time is also an applicative structure (see \cite[VI.2.9]{B85}). 

By using, say, Curry numerals, one obtains a combinator $\mb d$ for this representation of natural numbers. So, every pca can be turned 
into a pca over $\omega$ by using Curry numerals. On the other hand, the notion of pca over $\omega$ allows for other possible representations of natural numbers. Note that the existence of a combinator $\mb d$ for a given representation of natural 
numbers (together with a predecessor combinator), entails the existence of a primitive recursion operator $\mb r$ for such representation, that is, an element $\mb r$ such that:
\begin{align*}
\mb rab\bar 0&\simeq a;\\
\mb rab\overline{n+1}&\simeq b(\mb rab\bar n)\bar n.
\end{align*}
\end{remark}

\section{The theory $\czf$}
The logic of $\czf$ (Constructive Zermelo-Fraenkel set theory) is intuitionistic first order logic with equality. The only nonlogical symbol is $\in$ as in classical Zermelo-Fraenkel set theory $\sf ZF$. 
\begin{center} Axioms \end{center}

1. \textbf{Extensionality}:  $\forall x\, \forall y\, (\forall z\, (z\in x\biimp z\in y)\imp x=y)$,

2. \textbf{Pairing}: $\forall x\, \forall y\, \exists z\, (x\in z\land y\in z)$,

3. \textbf{Union}:  $\forall x\, \exists y\, \forall u\, \forall z\, (u\in z\land z\in x\imp u\in y)$,

4. \textbf{Infinity}:  $\exists x\, \forall y\, (y\in x\biimp y=0\lor \exists z\in x\, (y=z\cup\{z\}))$, 

5. \textbf{Set induction}:  $\forall x\, (\forall y\in x\, \vp(y)\imp \vp(x))\imp \forall x\, \vp(x)$,   for all formulae $\vp$,

6. \textbf{Bounded separation}:   $\forall x\, \exists y\, \forall z\, (z\in y\biimp z\in x\land \vp(z))$, for  $\vp$ bounded, where a formula is bounded if all quantifiers appear in the form $\forall x\in y$ and $\exists x\in y$,

7. \textbf{Strong collection}:    $\forall u\in x\, \exists v\, \vp(u,v)\imp \exists y\, (\forall u\in x\, \exists v\in y\, \vp(u,v)\land \forall v\in y\, \exists u\in x\, \vp(u,v))$,  for all formulae  $\vp$,

8. \textbf{Subset collection}: $\forall x\, \forall y\, \exists z\, \forall p\, (\forall u\in x\, \exists v\in y\, \vp(u,v,p)\imp \exists  q\in z\,  (\forall u\in x\, \exists v\in q\, \vp(u,v,p)\land \forall v\in q\, \exists u\in x\, \vp(u,v,p)))$,  for all  formulae $\vp$.

\begin{notation} Let $x=0$ be $\forall y\in x\, \neg (y=y)$ and $x=y\cup\{y\}$ be $\forall z\in x\, (z\in y\lor z=y)\land \forall z\in y\, (z\in x)\land y\in x$.
\end{notation}

\section{Finite types and axiom of choice}

Finite types $\sigma$ and their associated extensions $F_\sigma$ are defined by the following clauses:
\begin{itemize}
	\item $o\in\ft$ and $F_o=\omega$;
	\item if $\sigma,\tau\in\ft$, then $(\sigma)\tau\in\ft$ and \[F_{(\sigma)\tau}=F_\sigma\to F_\tau=\{ \text{total functions from $F_\sigma$ to $F_\tau$}\}.\]
\end{itemize} For brevity we write $\sigma\tau$ for $(\sigma)\tau$, if the type $\sigma$ is written as a single symbol. We say that $x\in F_\sigma$ has type $\sigma$. 

The set $\ft$ of all finite types, the set  $\{ F_\sigma\colon \sigma\in\ft\}$, and the set $\F=\bigcup_{\sigma\in\ft}F_\sigma$ all exist in $\czf$.

\begin{definition}[Axiom of choice in all finite types]
The schema $\ac_{\ft}$ consists of formulae
\[   \tag{$\ac_{\sigma,\tau}$}  \forall x^\sigma\, \exists y^\tau\, \vp(x,y)\imp \exists f^{\sigma\tau}\, \forall x^\sigma\,  \vp(x,f(x)), \]
where $\sigma$ and $\tau$ are (standard) finite types. 
\end{definition}

\begin{notation} We write $\forall  x^\sigma\, \vp(x)$ and $\exists x^\sigma\, \vp(x)$  as a shorthand for   $\forall x\, (x\in F_\sigma\imp  \vp(x))$ and $\exists x\, (x\in F_\sigma\land \vp(x))$ respectively. 
\end{notation}

\section{Defining extensional realizability in $\czf$}

In $\czf$, given a pca $A$ over $\omega$, we inductively define a class $\va$  such that \[\forall x\, (x\in\va\biimp x\subseteq A\times A\times \va).\]
The intuition for $\pair{a,b,y}\in x$ is that $a$ and $b$  are \emph{equal realizers} of the fact that  $y^A\in x^A$, where $x^A=\{y^A\colon \pair{a,b,y}\in x\text{ for some } a,b\in A\}$. 

General information on how to handle inductive definitions in $\czf$ can be found in \cite{A86,AR01,czf2}.  The inductive definition of $\va$ within $\czf$ is on par with that of $\mathrm{V}(A)$, the specifics of which appear in \cite[3.4]{R06}.

\begin{notation} We use $(a)_i$ or simply $a_i$ for $\mb {p_i}a$.   Whenever we write an application term $t$, we assume that it is defined. In other words, a formula $\vp(t)$ stands for $\exists a\, (t\simeq a\land \vp(a))$.
\end{notation}

\begin{definition}[Extensional realizability] We define the relation $a=b\fo \vp$, where $a,b\in A$ and $\vp$ is a realizability formula with parameters in $\va$. The atomic cases fall under the scope of definitions by transfinite recursion. 
\begin{align*}   
a=b&\fo x\in y && \Leftrightarrow && \exists z\, (\langle (a)_0,(b)_0,z\rangle \in y\land (a)_1=(b)_1\fo x=z)\\
a=b& \fo x=y && \Leftrightarrow  &&\forall \langle c,d,z\rangle \in x\, ((ac)_0=(bd)_0\fo z\in y) \text{ and }  \\
&&&&&  \forall \langle c,d,z\rangle \in y\, ((ac)_1=(bd)_1\fo z\in x)\\
a=b& \fo \vp\land \psi && \Leftrightarrow && (a)_0=(b)_0\fo \vp \land (a)_1=(b)_1\fo \psi \\
a=b& \fo \vp\lor\psi &&  \Leftrightarrow && (a)_0\simeq(b)_0\simeq \mb 0\land (a)_1=(b)_1\fo \vp \text{ or } \\ 
&&&&&  (a)_0\simeq (b)_0\simeq \mb 1\land (a)_1=(b)_1\fo \psi \\
a=b&\fo \neg\vp && \Leftrightarrow && \forall c, d\, \neg (c=d\fo \vp) \\
a=b&\fo \vp\imp\psi && \Leftrightarrow && \forall c,d\,  (c=d\fo \vp\imp  ac=bd\fo \psi) \\
a=b& \fo \forall x\in y\, \vp && \Leftrightarrow && \forall \langle c,d,x\rangle\in y\, (ac=bd\fo \vp) \\
a=b&\fo \exists x\in y\, \vp && \Leftrightarrow && \exists x\, (\langle (a)_0,(b)_0,x\rangle \in y\land (a)_1=(b)_1\fo \vp)\\
a=b& \fo \forall x\, \vp && \Leftrightarrow && \forall x\in \va\, (a=b\fo \vp) \\
a=b& \fo \exists x\, \vp && \Leftrightarrow && \exists x\in \va\, (a=b\fo \vp)
\end{align*}
\end{definition}
\begin{notation} We write $a\fo \vp$ for $a=a\fo \vp$. 
\end{notation}

The above definition builds on the variant \cite{R06} of   \McCarty realizability \cite{M84}, where bounded quantifiers are treated as quantifiers in their own right. Note that in the language of $\czf$, bounded quantifiers can be seen as  syntactic sugar by letting  $\forall x\in y\, \vp:=\forall x\, (x\in y\imp \vp)$ and $\exists x\in y\, \vp:=\exists x\, (x\in y\land \vp)$.  Nothing gets lost in translation, thanks to the following.
\begin{lemma}
	There are closed application terms $\mb u$ and $\mb v$ such that $\czf$ proves
	\[ \mb u\fo \forall x\in y\, \vp \biimp \forall x\, (x\in y\imp \vp), \]
	\[ \mb v\fo \exists x\in y\, \vp \biimp \exists x\, (x\in y\land \vp). \]
\end{lemma}
The advantage of having special clauses for bounded quantifiers is that it simplifies a great deal the construction of realizers. 

\begin{remark}
In the context of (finite type) arithmetic,  extensional notions of realizability typically give rise to a partial equivalence relation. Namely,  for every formula $\vp$,  the relation
$\{(a,b)\in A^2\colon a=b\fo \vp\}$ is  symmetric and transitive. This is usually  seen by induction on $\vp$, the atomic case being trivial. 
The situation, though,  is somewhat different in set theory. Say that $a=b\fo x\in y$ and $b=c\fo x\in y$. All we know is that for some $u,v\in\va$ we have that $\pair{(a)_0,(b)_0,u}, \pair{(b)_0,(c)_0,v}\in y$, $(a)_1=(b)_1\fo x=u$, and $(b)_1=(c)_1\fo x=v$. Since $u$ and $v$ need not be the same set, even if elements of $\va$ behave as expected, that is, $\{(a,b)\colon \pair{a,b,y}\in x\}$ is symmetric and transitive for any given $x,y\in\va$,\footnote{One could inductively define $\va$ so as to make $\{(a,b)\in A^2\colon \pair{a,b,y}\in x\}$
symmetric and transitive. Just let $x\in \va$ if and only if 
	\begin{itemize}
		\item $x$ consists of triples $\pair{a,b,y}$ with $y\in \va$;
		\item whenever $\pair{a,b,y}\in x$, $\pair{b,a,y}\in x$;
		\item whenever $\pair{a,b,y}\in x$ and $\pair{b,c,y}\in x$, also $\pair{a,c,y}\in x$.
\end{itemize}} we cannot conclude that $a=c\fo x\in y$. So, transitivity can fail. 

As it turns out, for our purposes, this is not an issue at all.  Note however that the canonical names for objects of finite type do indeed behave as desired and so does the relation $a=b\fo \vp$ for formulas of finite type arithmetic. This is in fact key in validating the axiom of choice in all finite types (Section \ref{finite type}). Except for this deviation, the clauses for connectives and quantifiers follow the general blueprint of extensional realizability.  We  just feel justified in keeping the notation $a=b\fo \vp$.
\end{remark}

\section{Soundness for intuitionistic first order logic with equality}

From now on, let $A$ be a pca over $\omega$ within $\czf$.  Realizability of the equality axioms relies on the following fact about pca's.

\begin{lemma}[Double recursion theorem]
There are  combinators $\mb g$  and $\mb h$ such that, for all $a,b,c\in A$:
\begin{itemize}
	\item $\mb gab\downarrow$ and $\mb hab\downarrow$;
	\item $\mb gabc\simeq a(\mb hab)c$;
	\item $\mb habc\simeq b(\mb gab)c$.
\end{itemize} 
\end{lemma}
\begin{proof}
	Let $t(a,b):=\lambda xc.a(\lambda c.bxc)c$. Set $\mb g:=\lambda ab.\mb ft(a,b)$, where $\mb f$ is the fixed point operator from the recursion theorem. Set $\mb h:=\lambda abc.b(\mb ft(a,b))c$. Verify that $\mb g$ and $\mb h$ are as desired.
\end{proof}

\begin{lemma}\label{equality}
There are closed application terms $\mb{i_r}$, $\mb{i_s}$, $\mb{i_t}$, $\mb{i_0}$ and $\mb{i_1}$ such that $\czf$ proves, for all $x,y,z\in\va$,
\begin{enumerate}[\quad $(1)$]
\item $\mb{i_r}\fo x=x$;
\item $\mb{i_s}\fo x=y\imp y=x$;
\item $\mb{i_t}\fo x=y\land y=z\imp x=z$;
\item $\mb{i_0}\fo x=y\land y\in z\imp x\in z$;
\item $\mb{i_1}\fo x=y\land z\in x\imp  z\in y$.
\end{enumerate}
\end{lemma}
\begin{notation} Write, say, $a_{ij}$ for $\mb{p_j}(\mb{p_i} a)$.
\end{notation}
\begin{proof}
(1) By the recursion theorem in $A$, we can find $\mb{i_r}$ such that
\[  \mb{i_r}a\simeq \ap(\ap a\mb{i_r})(\ap a\mb{i_r}). \]
By  set induction, we show that $\mb{i_r}\fo x=x$ for every $x\in\va$. Let $\pair{a,b,y}\in x$. We want $(\mb{i_r}a)_0=(\mb{i_r}b)_0\fo y\in x$. Now $(\mb{i_r}a)_{00}\simeq a$ and similarly for $b$. On the other hand, $(\mb{i_r}a)_{01}\simeq(\mb{i_r}b)_{01}\simeq \mb{i_r}$. By induction, $\mb{i_r}\fo y=y$, and so we are done. Similarly for $(\mb{i_r}a)_1=(\mb{i_r}b)_1\fo y\in x$.

(2) We just need to interchange. Let 
\[    \mb{i_s}:=\lambda ac.\ap (ac)_1(ac)_0. \]
Suppose $a=b\fo x=y$. We want $\mb{i_s}a=\mb{i_s}b\fo y=x$. Let $\pair{c,d,z}\in y$. By definition, $(ac)_1=(bd)_1\fo z\in x$. Now $(ac)_1\simeq (\mb{i_s}ac)_0$, and similarly $(bd)_1\simeq (\mb{i_s}bd)_0$. Then we are done. Similarly for the other direction. 

(3,4) Combinators $\mb{i_t}$ and $\mb{i_0}$ are defined by a double recursion  in $A$.  By induction on triples $\pair{x,y,z}$, one then shows that $\mb{i_t}\fo x=y\land y=z\imp x=z$ and $\mb{i_0}\fo x=y\land y\in z\imp x\in z$. Eventually,  $\mb {i_t}$ and $\mb {i_r}$ are solutions of equations of the form
\begin{align*}
	\mb {i_t}a& \simeq \mb t \mb {i_0}a, \\
	\mb {i_0}a&\simeq \mb r\mb {i_t}a,
\end{align*}
where $\mb t$ and $\mb r$ are given closed application terms. These are given by the fixed point operators from the double recursion theorem. 

(5) Set
\[  \mb{i_1}:=\lambda a.\ap (a_0a_{10})_{00}(\mb{i_t}(\ap a_{11}(a_0a_{10})_{01})). \]
\end{proof}

\begin{theorem}\label{int sound}
For every formula  $\vp(x_1,\ldots,x_n)$ provable in intuitionistic first order logic with equality, there exists a  closed application term $\mb e$ such that $\czf$ proves $\mb e\fo \forall x_1\cdots \forall x_n\, \vp(x_1,\ldots,x_n)$.
\end{theorem}
\begin{proof}
The proof is similar to \cite[5.3]{M84} and \cite[4.3]{R06}. 
\end{proof}

\section{Soundness for $\czf$}

We start with a  lemma concerning bounded separation. 

\begin{lemma}[$\czf$]\label{bounded}
	Let $\vp(u)$ be a bounded formula with parameters from $\va$ and $x\subseteq\va$. 
	Then
	\[  \{\pair{a,b,u}\colon a,b\in A\land u\in x\land a=b\fo \vp(u)\} \]
	is a set. 
\end{lemma}
\begin{proof}
	As in \cite[Lemma 4.5, Lemma 4.6, Corollary 4.7]{R06}.
\end{proof}

\begin{theorem}\label{czf sound}
For every theorem $\vp$ of $\czf$, there is a closed application term $\mb e$ such that  $\czf$ proves $\mb e\fo \vp$. 
\end{theorem}

\begin{proof}
In view of Theorem \ref{int sound}, it is sufficient to show that every axiom of $\czf$ has a realizer. The proof is similar to that of \cite[Theorem 5.1]{R06}. The rationale is simple:  use the same  realizers, duplicate the names.  Remember that  $\mb e\fo\vp$ means $\mb e=\mb e\fo\vp$.

\textbf{Extensionality}. Let $x,y\in\va$.  Suppose $a=b\fo z\in x\biimp z\in y$ for all $z\in\va$. We look for $\mb e$ such that $\mb ea=\mb eb\fo x=y$. Set 
\[ \mb e:=\lambda ac.\ap(a_0(\ap c\mb{i_r}))(a_1(\ap c\mb{i_r})). \]
Suppose $\pair{c,d,z}\in x$. Then $\ap c\mb{i_r}=\ap d\mb{i_r}\fo z\in x$, since $\mb{i_r}\fo z=z$. Then $a_0(\ap c\mb{i_r})=b_0(\ap d\mb{i_r})\fo z\in y$. Therefore, $(\mb eac)_0=(\mb ebd)_0\fo z\in y$, as desired. The other direction is similar.\\

\textbf{Pairing}. Find $\mb e$ such that for all $x,y\in \va$,
\[   \mb e\fo x\in z\land y\in z, \]
for some $z\in \va$. Let $x,y\in\va$ be given. Define $z=\{\pair{\mb 0,\mb 0,x},\pair{\mb 0,\mb 0,y}\}$.
Let 
\[   \mb e=\ap (\ap \mb 0\mb{i_r})(\ap \mb 0\mb{i_r}). \]

\textbf{Union}. Find $\mb e$ such that for all $x\in \va$,
\[ \mb e\fo \forall u\in x\, \forall v\in u\, (v\in y), \]
for some $y\in \va$. Given $x\in\va$, let $y=\{\pair{c,d,v}\colon \exists \pair{a,b,u}\in x\, (\pair{c,d,v}\in u)\}$. Set $\mb e:=\lambda ac.\ap c\mb{i_r}$.\\

\textbf{Infinity}. Let $\dot\omega=\{\pair{\bar n,\bar n,\dot n}\colon n\in\omega\}$, where $\dot n=\{\pair{\bar m,\bar m, \dot m}\colon m<n\}$. Let us find $\mb e$ such that for all $y\in\va$,

\[ \mb e\fo y\in \dot \omega\biimp y=0\lor \exists z\in \dot\omega\, (y=z\cup\{z\}). \]
Recall that $y=0$ stands for $\forall x\in y\, \neg (x=x)$ and $y=z\cup\{z\}$ stands for $\forall x\in y\, (x\in z\lor x=z)\land (\forall x\in z\, (x\in y)\land z\in y)$. 

Let $\vartheta(y):=y=0\lor \exists z\in \dot\omega\, (y=z\cup\{z\})$. We want $\mb e$ such that for every $y\in\va$
\[ \tag{1} \mb e_0\fo y\in\dot\omega\imp \vartheta(y), \]
\[ \tag{2} \mb e_1\fo \vartheta(y)\imp y\in\dot\omega. \] 

Let us first consider (1). Suppose $a=b\fo y\in\dot \omega$. We want  $\mb e_0a=\mb e_0b\fo \vartheta(y)$.

By definition, there is $n\in\omega$ such that $a_0\simeq b_0\simeq \bar n$ and $a_1=b_1\fo y=\dot n$.

Case $n=0$. Then $\mb 0\fo y=0$, and so $\mb p\mb 0\mb 0\fo \vartheta(y)$. 

Case $n>0$.  We have $\pred a_0\simeq \pred b_0\simeq \bar m$ with $n=m+1$. We aim for a term $t(x)$ such that  $t(a)=t(b)\fo \exists z\in\dot\omega\, (y=z\cup\{z\})$ by requiring 
\[    t(a)_0\simeq t(b)_0\simeq \bar m, \]
\[   \tag{3} t(a)_1=t(b)_1\fo y=\dot m\cup\{\dot m\}. \]
If we succeed, then 
\[     \mb p\mb 1 t(a)=\mb p\mb 1t(b)\fo \vartheta(y). \]
Now, (3)  amounts to 
\begin{align}
	\tag{4} t(a)_{10}=t(b)_{10}&\fo \forall x\in y\, (x\in \dot m\lor x=\dot m)\\
	\tag{5} t(a)_{110}=t(b)_{110}&\fo \forall x\in \dot m\, (x\in y)\\
	\tag{6} t(a)_{111}=t(b)_{111}&\fo \dot m\in y
\end{align}
Part (4). Let $\pair{c,d,x}\in y$. Then $(a_1c)_0=(b_1d)_0\fo x\in \dot n$, that is,
\[   \pair{(a_1c)_{00},(b_1d)_{00},\dot k}\in \dot n, \]
\[   (a_1c)_{01}=(b_1d)_{01}\fo x=\dot k, \]
where $(a_1c)_{00}\simeq (b_1d)_{00}\simeq \bar k$. Here we have two more cases. If $k=m$, then 
\[ \mb p\mb 1  (a_1c)_{01}=\mb p\mb 1(b_1d)_{01}\fo x\in \dot m\lor x=\dot m. \]
If $k<m$, then $\pair{\bar k,\bar k,\dot k}\in\dot m$ and $\mb p\bar k(a_1c)_{01}=\mb p\bar k(b_1d)_{01}\fo x\in \dot m$, so that 
\[   \mb p\mb 0(\mb p\bar k(a_1c)_{01})=\mb p\mb 0(\mb p\bar k(b_1d)_{01}) \fo x\in\dot m\lor x=\dot m. \]
Then $t(a)$ such that 
\[   t(a)_{10}\simeq \lambda c.\mb d(a_1c)_{00}(\pred a_0)(\mb p\mb 1(a_1c)_{01})(\mb p\mb 0(a_1c)_0)\]
is as desired. 

Parts (5) and (6).  Let $t(a)$ satisfy 
\[   t(a)_{110}\simeq \lambda x.(a_1x)_1, \]
\[  t(a)_{111}\simeq (a_1(\pred a_0))_1. \]

We want $\mb e$ such that
\[   \mb e_0\simeq\lambda a.\mb d\mb 0a_0(\mb p\mb 0\mb 0)(\mb p\mb 1t(a)). \]
Then $\mb e_0$ does the job. 

As for (2),  suppose $a=b\fo \vartheta(y)$. We want $\mb e_1a=\mb e_1b\fo y\in \dot \omega$. By unravelling the definitions, we obtain two cases. 

(i) $a_0\simeq b_0\simeq \mb 0$ and $a_1=b_1\fo y=0$. It follows that $y=\dot 0$ and so $\mb{i_r}\fo y=\dot 0$. Therefore $\mb p a_0\mb{i_r}=\mb p b_0\mb{i_r}\fo y\in\dot\omega$, as $\pair{\mb 0,\mb 0,\dot 0}\in\dot\omega$.

(ii)  $a_0\simeq b_0\simeq \mb 1$ and $a_1=b_1\fo \exists z\in\dot \omega\, (y=z\cup\{z\})$  
Then there exists $m\in\omega$ such that $a_{10}\simeq b_{10}\simeq \bar m$ and 
\[  \tag{7} a_{11}=b_{11}\fo y=\dot m\cup\{\dot m\}. \]
We aim for a term $s(x)$ such that $s(a)=s(b)\fo y=\dot{n}$, where $n=m+1$. If we succeed, then
\[  \mb p(\succe a_{10})s(a)=\mb p(\succe b_{10})s(b) \fo y\in\dot\omega. \]
Note in fact that $\succe a_{10}\simeq \succe b_{10}\simeq \bar n$.  

For the left to right inclusion, suppose $\pair{c,d,x}\in y$. Our goal is $(s(a)c)_0=(s(b)d)_0\fo x\in\dot n$.  It follows from (7) that 
\[   a_{110}=b_{110}\fo \forall x\in y\, (x\in\dot m\lor x=\dot m), \]
and therefore
\[   \tag{8} a_{110}c=b_{110}d\fo x\in\dot m\lor x=\dot m. \]
>From (8) we get two more cases. First case: $(a_{110}c)_0\simeq (b_{110}d)_0\simeq \mb 0$ and $(a_{110}c)_1=(b_{110}d)_1\fo x\in\dot m$. Then one can verify that 
\[ (a_{110}c)_1=(b_{110}d)_1\fo x\in\dot n. \]
Second case: $(a_{110}c)_0\simeq (b_{110}d)_0\simeq \mb 1$ and $(a_{110}c)_1=(b_{110}d)_1\fo x=\dot m$. Then 
\[  \mb p\bar m(a_{110}c)_1=\mb p\bar m(b_{110}d)_1\fo x\in \dot n. \]
Let $s(x)$ be such that 
\[   (s(a)c)_0\simeq \mb d\mb 0(a_{110}c)_0(a_{110}c)_1(\mb p a_{10}(a_{110}c)_1). \]

For the right to left inclusion, suppose $k<n$. Our goal is $(s(a)\bar k)_1=(s(b)\bar k)_1\fo \dot k\in y$.  It follows from (7) that 
\begin{align*}
	\tag{9} a_{1110}=b_{1110}&\fo \forall x\in\dot m\, (x\in y), \\
	\tag{10} a_{1111}=b_{1111}&\fo \dot m\in y.
\end{align*}
If $k<m$, then $\pair{\bar k,\bar k, \dot k}\in \dot m$, and hence  $a_{1110}\bar k=b_{1110}\bar k\fo \dot k\in y$ by (9). On the other hand, if $k=m$ then (10) gives us the realizers. Therefore let $s(x)$ be such that 
\[   (s(a)\bar k)_1\simeq \mb d \bar ka_{10}a_{1111}(a_{1110}\bar k).   \]

We thus want $\mb e$ such that  
\[  \mb e_1\simeq \lambda a.\mb d\mb 0a_0(\mb pa_0\mb{i_r})(\mb p(\succe a_{10})s(a)).\]
Then $\mb e_1$ does the job. \\

\textbf{Set induction}. By the recursion theorem, let $\mb e$ be  such that $\mb ea\simeq a(\lambda c.\mb ea)$. 
Prove that 
\[ \mb e\fo \forall x\, (\forall y\in x\, \vp(y)\imp \vp(x))\imp \forall x\, \vp(x). \]
Let $a=b\fo \forall x\, (\forall y\in x\, \vp(y)\imp \vp(x))$. By definition, $a=b\fo \forall y\in x\, \vp(y)\imp \vp(x)$ for every $x\in\va$. By set induction, we show that   $\mb ea=\mb eb\fo \vp(x)$ for every $x\in\va$. Assume by induction  that  $\mb ea=\mb eb\fo \vp(y)$ for every $\pair{c,d,y}\in x$. This means that $\lambda c.\mb ea=\lambda d.\mb eb\fo \forall y\in x\, \vp(y)$. Then $a(\lambda c.\mb ea)=b(\lambda d.\mb eb)\fo \vp(x)$. The conclusion $\mb ea=\mb eb\fo \vp(x)$ follows.  \\

 \textbf{Bounded separation}. Find $\mb e$ such that for all $x\in\va$,
\[  \mb e\fo \forall u\in y\, (u\in x\land \vp(u))\land \forall u\in x\, (\vp(u)\imp u\in y), \]
for some $y\in\va$. Given $x\in\va$, let 
\[ y=\{ \pair{\ap ac,\ap bd,u}\colon \pair{a,b,u}\in x\land c=d\fo \vp(u)\}. \]
It follows from Lemma \ref{bounded} that  $y$ is a set. Moreover, $y$ belongs to $\va$. We want $\mb e$ such that 
\begin{align*}
\mb e_0&\fo \forall u\in y\, (u\in x\land \vp(u)),\\
\mb e_1&\fo \forall u\in x\, (\vp(u)\imp u\in y).
\end{align*}
By letting $\mb e=\mb p e_0e_1$, where  
\begin{align*}
e_0&:=\lambda f.\mb p(\mb p f_0\mb{i_r})f_1, \\
e_1&:= \lambda ac.\mb p(\mb pac)\mb{i_r},
\end{align*}
one verifies that $\mb e$ is as desired. \\

 \textbf{Strong Collection}. Set $\mb e:=\lambda a.\ap(\lambda c.\ap c(ac))(\lambda c.\ap c(ac))$.
Let $a=b\fo \forall u\in x\, \exists v\, \vp(u,v)$. By strong collection, we can find a set $y$ such that 
\begin{itemize}
	\item $\forall \pair{c,d,u}\in x\, \exists v\in\va\, (\pair{c,d,v}\in y\land  ac=bd\fo \vp(u,v))$, and 
	\item $\forall z\in y\, \exists \pair{c,d,u}\in x\, \exists v\in\va\, (z=\pair{c,d,v}\land ac=bd\fo \vp(u,v))$.
\end{itemize}
In particular, $y\in\va$. Show that 
\[  \mb ea=\mb eb\fo \forall u\in x\, \exists v\in y\, \vp(u,v)\land \forall v\in y\exists u\in x\, \vp(u,v). \]

\textbf{Subset collection}. We look for $\mb e$ such that for all $x,y\in\va$ there is a $z\in\va$ such that for all $p\in\va$
\[  \mb e\fo \forall u\in x\, \exists v\in y\, \vp(u,v,p)\imp \exists q\in z\, \psi(x,q,p), \]
where 
\[ \psi(x,q,p):= \forall u\in x\, \exists v\in q\, \vp(u,v,p)\land \forall v\in q\, \exists u\in x\, \vp(u,v,p). \]
Form the set $y'=\{\pair{f,g,v}\colon f,g\in A\land \exists i,j\in A\, \pair{i,j,v}\in y\}$. By subset collection, we can find a set $z'$ such that for all $a,b,p$, if 
\[  \tag{11} \forall \pair{c,d,u}\in x\, \exists \pair{\ap ac,\ap bd,v}\in y'\, (ac)_1=(bd)_1\fo \vp(u,v,p), \]
then there is a $q\in z'$ such that 
\[ \tag{12} \forall \pair{c,d,u}\in x\, \exists w\in q\, \vartheta\land \forall w\in q\, \exists \pair{c,d,u}\in x\, \vartheta,  \]
where $\vartheta=\vartheta(c,d,u,w;a,b,p)$ is 
\[\exists v\, (w=\pair{\ap ac,\ap bd,v}\land (ac)_1=(bd)_1\fo \vp(u,v,p)). \] 

Note that the $q\in z'$ asserted to exist is a subset of $y'$ and so $q\in\va$. On the other hand, there might be $q\in z'$ that are not in $\va$, and hence $z'$ need not be a subset of $\va$.  Let $z''=\{q\cap y'\colon q\in z'\}$. Now, $z''\subseteq\va$.  Finally, set
\[   z=\{\pair{\mb 0,\mb 0, q}\colon q\in z''\}. \]
Then $z\in\va$. It remains to find $\mb e$. Let $p\in\va$ and suppose 
\[ \tag{13} a=b\fo \forall u\in x\, \exists v\in y\, \vp(u,v,p). \] We would like to have
\[      \mb ea=\mb eb\fo \exists q\in z\, \psi(x,q,p). \]
By definition of $z$, we let  $(\mb ea)_0\simeq \mb 0$ and we look for a $q\in z''$ such that $(\mb ea)_1=(\mb eb)_1\fo \psi(x,q,p)$, that is,
\begin{align*}
(\mb ea)_{10}=(\mb eb)_{10}&\fo \forall u\in x\, \exists v\in q\, \vp(u,v,p), \\
(\mb ea)_{11}=(\mb eb)_{11}&\fo  \forall v\in q\, \exists u\in x\, \vp(u,v,p).
\end{align*}  
By (13) one can see that the parameters $a,b,p$ satisfy (11). Let $q\in z'$ be as in (12). We have already noticed that $q\in z''$. Let $\mb e$ be such that 
\begin{align*}
(\mb ea)_{10}&\simeq \lambda c.\mb p(\mb pac)(ac)_1, \\
(\mb ea)_{11}&\simeq \lambda f. \mb pf_1(f_0f_1)_1.
\end{align*}
One can verify that $\mb e$ is as desired. 
\end{proof}

\section{Realizing the axiom of choice in all finite types}\label{finite type}

We will make use of certain canonical names for pairs in $\va$.

\begin{definition}[Internal pairing]
For $x,y\in\va$, let
\[  \vset{x}=\{\pair{\mb 0,\mb 0,x}\}, \]
\[   \vset{x,y}=\{\pair{\mb 0,\mb 0,x},\pair{\mb 1,\mb 1,y}\}, \]
\[ \vpair{x,y}=\{\pair{\mb 0,\mb 0,\vset{x}}, \pair{\mb 1,\mb 1,\vset{x,y}}\}. \]

Note that all these sets are in $\va$.
\end{definition}

Below we shall use $\mathrm{UP}(x,y,z)$  and $\mathrm{OP}(x,y,z)$ as abbreviations for the set-theoretic formulae expressing, respectively,  that $z$ is the unordered pair of $x$ and $y$ (in standard  notation, $z=\{x,y\}$) and $z$  is the ordered pair of $x$ and $y$ (in standard  notation, $z=\pair{x,y}$). E.g.,  $\mathrm{UP}(x,y,z)$  stands for $x\in z\land y\in z\land \forall u\in z\, (u=x\lor u=y)$. Similarly, one can pick a suitable rendering of  $\mathrm{OP}(x,y,z)$  according to the definition of ordered pair $\pair{x,y}:=\{\{x\},\{x,y\}\}$.

\begin{lemma}\label{pairs}
There are closed application terms $\mb{u_0}$, $\mb{u_1}$, $\mb v$, $\mb w$, $\mb z$ such that for all $x,y\in \va$
\begin{align*}
\mb{u_0}&\fo \mathrm{UP}(x,x,\vset{x}), \\
\mb{u_1}&\fo \mathrm{UP}(x,y,\vset{x,y}),  \\
\mb v &\fo \mathrm{OP}(x,y, \vpair{x,y}), \\
\mb w& \fo \vpair{x,y}=\vpair{u,v}\imp x=u\land y=v,\\
\mb z&\fo \mathrm{OP}(x,y,z) \imp z=\vpair{x,y}.
\end{align*}
\end{lemma}
\begin{proof}
This is similar to \cite[3.2, 3.4]{M84}. 	
\end{proof}

We now build a copy of  the hereditarily effective operations  relative to a pca $A$. 
\begin{definition}[$\mathsf{HEO}_A$] Let $A$ be a pca over $\omega$ with map $n\mapsto \bar n$  from $\omega$ to $A$. For any finite type $\sigma$, we define $a=_\sigma b$ with $a,b\in A$ by letting:
\begin{itemize}
	\item $a=_0 b$ iff there is $n\in\omega$ such that $a=b=\bar n$; 
	\item $a=_{\sigma\tau} b$ iff for every $c=_{\sigma}d$ we have $ac=_\tau bd$.
\end{itemize}
Let $A_\sigma=\{a\in A\colon a=_\sigma a\}$.
\end{definition}

\begin{lemma} 
For any type $\sigma$, and for all $a,b,c\in A$:
\begin{itemize}
	\item if $a=_\sigma b$ and $b=_\sigma c$, then $a=_\sigma a$, $b=_\sigma a$, and $a=_\sigma c$.
\end{itemize} 
It thus follows that
$A_\sigma=\bigcup_{b\in A}\{a\in A\colon a=_\sigma b\}=\bigcup_{a\in A}\{b\in A\colon a=_\sigma b\}$ and $=_\sigma$ is an equivalence relation on $A_\sigma$.
\end{lemma}
\begin{proof}
By induction on the type. 	
\end{proof} 

 \begin{definition}[Internalization of objects of finite type]
For $a\in A_\sigma$, we define $\tuep a\sigma\in\va$    as follows:
\begin{itemize}
	\item if $a=\bar n$, let $\tuep ao =\{ \pair {\bar m,\bar m,\tuep{\bar{m}}o}\colon m<n\}$;
	\item if $a\in A_{\sigma\tau}$, let $\tuep a{\sigma\tau}=\{ \pair{c,d,\vpair{c^\sigma,\tuep e\tau}}     \colon c=_\sigma d\text{ and } ac\simeq e\}$. 
\end{itemize}

Finally, for any finite type $\sigma$, let 
\[  \dot F_\sigma=\{ \pair{a,b,\tuep a{\sigma}}\colon a=_\sigma b\} \]
be our name for $F_\sigma$.
\end{definition}

Note that $\dot F_o=\dot \omega$, where $\dot\omega$ is the name for $\omega$ used to realize the infinity axiom in the proof of Theorem \ref{czf sound}. \\

\begin{notation} Write $\fo \vp$ for $\exists a, b\in A\, (a=b\fo \vp)$.
\end{notation}

\begin{lemma}[Absoluteness and uniqueness up to extensional equality]\label{abs}
 For all $a,b\in A_\sigma$,
\begin{itemize}
	\item $\fo \tuep a\sigma=\tuep b\sigma$ implies $a=_\sigma b$,
	\item $a=_\sigma b$ implies $\tuep a\sigma=\tuep b\sigma$. 
\end{itemize}
\end{lemma}
\begin{proof}
By induction on the type. 

Type $o$. Let $a=\bar n$ and $b=\bar m$ with $n,m\in\omega$. Suppose $\fo  \tuep a{o}=\tuep bo$. By a double arithmetical induction one shows  $n=m$.  The second part is obvious as $a=_ob$ implies $a=b$.

Type $\sigma\tau$. Let $a,b\in A_{\sigma\tau}$. Suppose $\fo \tuep a{\sigma\tau}=\tuep b{\sigma\tau}$. The aim is to show that $a=_{\sigma\tau} b$.
Let $c\in A_{\sigma}$ and  $ac\simeq e$. Then $\fo \vpair{\tuep c\sigma,\tuep e\tau}\in \tuep a{\sigma\tau}$ and hence  $\fo \vpair{\tuep c\sigma,\tuep e\tau}\in \tuep b{\sigma\tau}$.
>From the latter we infer that there exist $c_0\in A_{\sigma}$ and $e_0\in A_{\tau}$ such that $bc_0\simeq e_0$ and $\fo \vpair{\tuep c\sigma,\tuep e\tau}=\vpair{\tuep {c_0}\sigma,\tuep {e_0}\tau}$. By the properties of internal pairing, we obtain $\fo \tuep c\sigma=\tuep {c_0}\sigma\;\wedge\; \tuep e\tau = \tuep {e_0}\tau$ giving
$c=_{\sigma}c_0$ and $e=_{\tau} e_0$ by the induction hypothesis. Whence $ac=_{\tau}bc_0=_{\tau}bc$ as $b \in A_{\sigma\tau}$. As a result one has 
$ac=_{\tau}bd$ whenever $c=_{\sigma}d$, yielding $a=_{\sigma\tau}b$. 

For the second part, suppose $a=_{\sigma\tau} b$. An element of $\tuep a{\sigma\tau}$ is of the form $\pair{c,d,\vpair{\tuep c{\sigma},\tuep e\tau}}$ where $c=_{\sigma}d$ and $ac\simeq e$. Let $e_0\simeq bc$. As $ac=_{\tau}bc$ the induction hypothesis yields $\tuep e\tau =\tuep {e_0}\tau$, and hence 
$\pair{c,d,\vpair{\tuep c{\sigma},\tuep e\tau}}=\pair{c,d,\vpair{\tuep c{\sigma},\tuep {e_0}\tau}}\in \tuep b{\sigma\tau}$, showing  $\tuep a{\sigma\tau}\subseteq\tuep b{\sigma\tau}$.
Owing to the symmetry of the argument, we can conclude that  $\tuep a{\sigma\tau}=\tuep b{\sigma\tau}$.
\end{proof}

\begin{theorem}[Choice]\label{choice}
There exists a closed application term $\mb e$ such that  $\czf$ proves
\[  \mb e\fo \forall x\in \dot F_\sigma\, \exists y\in \dot F_\tau\, \vp(x,y)\imp \exists f\colon \dot F_\sigma\to \dot F_\tau\, \forall x\in \dot F_\sigma\, \vp(x,f(x)), \]
for all finite types $\sigma$ and $\tau$ and for every formula $\vp$.
\end{theorem}
\begin{proof}
Suppose $a=b\fo  \forall x\in \dot F_\sigma\, \exists y\in \dot F_\tau\, \vp(x,y)$. By definition, this means that for every $\pair{c,d,\tuep c\sigma}\in \dot F_\sigma$  we have
\[ \tag{1} \pair{(ac)_0,(bd)_0,\tuep e\tau}\in \dot F_\tau, \] 
\[ \tag{2}	(ac)_1=(bd)_1\fo \vp(\tuep c\sigma,\tuep e\tau), \]
where $e\simeq (ac)_0$. 
Let 
\[  f=\{\pair{c,d,\vpair{\tuep c\sigma,\tuep e\tau}}\colon c=_\sigma d\land e\simeq (ac)_0\}. \]
Note that $c=_\sigma d$ implies $(ac)_0\downarrow$ by (1). 

Below we shall use $z=\langle x,y\rangle$ as a somewhat sloppy abbreviation for $\mathrm{OP}(x,y,z)$. We  look for an $\mb e$ such that 
\[ \tag{3}  (\mb ea)_{0}=(\mb eb)_0\fo \forall z\in f\, \exists x\in \dot F_\sigma\, \exists y\in \dot F_\tau\, (z=\pair{x,y}),  \]
\[\tag{4}  (\mb ea)_{10}=(\mb e b)_{10}\fo  \forall x\in \dot F_\sigma\, \exists y\in \dot F_\tau\,  \exists z\in f\, (z=\pair{x,y}\land \vp(x,y)), \]
\[ \tag{5}  (\mb ea)_{11}=(\mb eb)_{11}\fo \forall z_0\in f\, \forall z_1\in f\, \forall x,y_0,y_1\, (z_0=\pair{x,y_0}\land z_1=\pair{x,y_1}\imp y_0=y_1). \]

First, note that $\lambda c.(ac)_0=_{\sigma\tau}\lambda d.(bd)_0$. This follows from (1). In fact, $c=_\sigma d$ implies $(ac)_0=_\tau (bd)_0$, for all $c,d\in A$. Moreover, since this is an equivalence relation, we have $\lambda c.(ac)_0\in A_{\sigma\tau}$. 

For (3), let $\mb e$ be such that 
\begin{align*}
((\mb e a)_0c)_0&\simeq c,  &   ((\mb e a)_0 c)_{10}&\simeq (ac)_0,  \\
& & ((\mb e a)_0c)_{11}&\simeq \mb v, 
\end{align*}
where $\mb v\fo \vpair{x,y}=\pair{x,y}$ for all $x,y\in\va$ as in Lemma \ref{pairs}. Let us show that any such $\mb e$ satisfies (3). Let $\pair{c,d,\vpair{\tuep c\sigma,\tuep e\tau}}\in f$, where $e\simeq (ac)_0$. We would like
\[  (\mb e a)_0c=(\mb e b)_0d\fo \exists x\in \dot F_\sigma\, \exists y\in \dot F_\tau\, (\vpair{\tuep c\sigma,\tuep e\tau}=\pair{x,y}). \]
Now, $\pair{c,d,\tuep c\sigma}\in \dot F_\sigma$, $c\simeq ((\mb e a)_0c)_0$, and $d\simeq ((\mb e b)_0d)_0$. Therefore, we just need to verify
\[  ((\mb ea)_0c)_1=((\mb eb)_0d)_1\fo \exists y\in \dot F_\tau\, \vpair{\tuep c\sigma,\tuep e\tau}=\pair{\tuep c\sigma,y}.\]
Similarly, $\pair{(ac)_0,(bd)_0,\tuep e\sigma}\in \dot F_\tau$ since, as noted before, $(ac)_0=_\tau (bd)_0$. On the other hand, $(ac)_0\simeq  ((\mb e a)_0 c)_{10}$ and $(bd)_0\simeq ((\mb eb)_0d)_{10}$. So we just need to show that 
\[  ((\mb ea)_0c)_{11}=((\mb eb)_0d)_{11}\fo  \vpair{\tuep c\sigma,\tuep e\tau}=\pair{\tuep c\sigma,\tuep e\tau}. \]
Now, $((\mb ea)_0c)_{11}\simeq ((\mb eb)_0d)_{11}\simeq \mb v$, and $\mb v\fo \vpair{\tuep c\sigma,\tuep e\tau}=\pair{\tuep c\sigma,\tuep e\tau}$. So we are done. 

As for (4), Let $\mb e$ be such that 
\begin{align*}
((\mb ea)_{10}c)_0& \simeq (ac)_0,  &   ((\mb ea)_{10}c)_{10}&\simeq (ac)_0, &   ((\mb ea)_{10}c)_{110}&\simeq \mb v, \\
&&&&   ((\mb ea)_{10}c)_{111}&\simeq (ac)_1,
\end{align*}
where $\mb v$ is as in part (3). That $\mb e$ satisfies (4) is proved in similar fashion by  using (1) and (2).

For (5), suppose $\pair{c_i,d_i,z_i}\in f$ with $z_i=\vpair{\tuep {c_i}\sigma,\tuep {e_i}\tau}$ and $e_i\simeq (ac_i)_0$, where $i=0,1$. We are looking for an $\mb e$ such that 
\[   (\mb ea)_{11}c_0c_1=(\mb eb)_{11}d_0d_1\fo  z_0=\pair{x,y_0}\land z_1=\pair{x,y_1}\imp y_0=y_1, \]
for all $x,y_0,y_1\in\va$. Suppose
\[ \tag{6}   g=h\fo z_0=\pair{x,y_0}\land z_1=\pair{x,y_1}. \]
We want $(\mb ea)_{11}c_0c_1g=(\mb eb)_{11}d_0d_1h\fo y_0=y_1$. Unravelling (6), we get 
\[ g_i=h_i\fo \vpair{\tuep {c_i}\sigma,\tuep {e_i}\tau}=\pair{x,y_i}.  \]
By Lemma \ref{pairs},
\[  \mb wg_i=\mb w h_i\fo \tuep {c_i}\sigma=x\land \tuep {e_i}\tau=y_i, \]
for some closed application term $\mb w$.  By the realizabilty of equality, it follows that 
\[ \tag{7} \fo \tuep {c_0}\sigma=\tuep {c_1}\sigma. \] 
Also,
\[  \mb p(\mb w g_0)_1(\mb wg_1)_1=\mb p(\mb w h_0)_1(\mb wh_1)_1\fo \tuep {e_0}\tau=y_0\land \tuep {e_1}\tau=y_1. \]
By absoluteness, (7) implies $c_0=_\sigma c_1$.  As $\lambda c.(ac)_0\in A_{\sigma\tau}$, we have $(ac_0)_0=_\tau (ac_{1})_0$, that is, $e_0=_\tau e_1$. By uniqueness, $\tuep {e_0}\tau=\tuep {e_1}\tau$. By realizability of equality, there is a closed application term $\mb i$ such that 
\[  \mb i\fo z=y_0\land z=y_1\imp y_0=y_1. \]
Therefore $\mb e$ can be chosen such that 
\[  (\mb ea)_{11}c_0c_1g\simeq \mb i(\mb p(\mb w g_0)_1(\mb wg_1)_1)\]
is as required.

By $\lambda$-abstraction, one can find $\mb e$ satisfying (3), (4), and (5). 
\end{proof}

\begin{theorem}[Arrow types]\label{arrow}
There exists a closed application term $\mb e$ such that $\czf$ proves
	\[   \mb e\fo \dot F_{\sigma\tau}= \dot F_\sigma\to \dot F_\tau,  \]
for all finite types $\sigma$ and $\tau$.
\end{theorem}
\begin{proof}
We look for $\mb e$ such that  
\[  \mb e_0\fo \forall f\in \dot F_{\sigma\tau}\, (f\colon\dot F_\sigma\to \dot F_\tau), \]
and for every $f\in \va$,
\[  \mb e_1\fo (f\colon \dot F_\sigma\to \dot F_\tau) \imp f\in \dot F_{\sigma\tau}. \]
	
For $\mb e_0$, we need that for all $a=_{\sigma\tau} b$,
\[ \tag{1} (\mb e_0 a)_0=(\mb e_0 b)_0 \fo \forall z\in  \tuep a{\sigma\tau}\, \exists x\in\dot F_\sigma\, \exists y\in\dot F_\tau\, (z=\pair{x,y}),\]
\[ \tag{2} (\mb e_0a)_{10}=(\mb e_0b)_{10}\fo \forall x\in\dot F_\sigma\, \exists y\in\dot F_\tau\, \exists z\in  \tuep a{\sigma\tau}\, (z=\pair{x,y}),\]
\[ \tag{3} (\mb e_0 a)_{11}=(\mb e_0 b)_{11}\fo \forall z_0\in \tuep a{\sigma\tau}\, \forall z_1\in \tuep a{\sigma\tau}\, \forall x,y_0,y_1\, (z_0=\pair{x,y_0}\land z_1=\pair{x,y_1}\imp y_0=y_1). \]

For (1), let $\mb e_0$ be such that 
\[ (\mb e_0a)_0\simeq\lambda c.\mb p c(\mb p (ac)\mb v), \]
where $\mb v\fo \vpair{x,y}=\pair{x,y}$ for all $x,y\in \va$ as in Lemma \ref{pairs}.
	
Let us verify that $\mb e_0$ does the job. Let $a=_{\sigma\tau}b$. 
We want to show
\[    \lambda c.\ap  c(\mb p (ac)\mb v)=\lambda d.\ap  d(\mb p (bd)\mb v)\fo \forall z\in \tuep a{\sigma\tau}\, \exists x\in\dot F_\sigma\, \exists y\in\dot F_\tau\, (z=\pair{x,y}). \]
Let $\pair{c,d,\vpair{ \tuep c{\sigma}, \tuep c{\tau}}}\in  \tuep a{\sigma\tau}$, where $c=_\sigma d$ and $ac\simeq e$.  We want
\[   \ap c(\mb p (ac)\mb v)=\ap d(\mb p (bd)\mb v)\fo \exists x\in \dot F_\sigma\, \exists y\in\dot F_\tau\, (\vpair{ \tuep c{\sigma}, \tuep c{\tau}}=\pair{x,y}). \]
By definition, $\pair{c,d, \tuep c{\sigma}}\in\dot F_\sigma$. Let us check that
	\[ \mb p (ac)\mb v=\ap (bd)\mb v\fo  \exists y\in\dot F_\tau\, (\vpair{ \tuep c{\sigma}, \tuep c{\tau}}=\pair{ \tuep c{\sigma},y}). \]
We have $ac=_\tau bd$ and hence $\pair{ac,bd, \tuep c{\tau}}\in \dot F_\tau$.  
Finally, 
\[ \mb v\fo \vpair{ \tuep c{\sigma}, \tuep c{\tau}}=\pair{ \tuep c{\sigma}, \tuep c{\tau}}. \]
	
For (2), let $\mb e_0$ be such that 
\[   (\mb e_0a)_{10}\simeq\lambda x.\ap (ax)(\ap x\mb v), \]
where $\mb v$ is as above.

For (3), let $\mb e_0$ be such that 
\[  (\mb e_0 a)_{11}c_0c_1g\simeq \mb i (\mb p(\mb w g_0)_1(\mb wg_1)_1),\]
where $\mb w$ and $\mb i$ are as in the proof of Theorem \ref{choice}. \\
	
As for $\mb e_1$, suppose that $f\in\va$ and
	\[   a=b\fo f\colon\dot F_\sigma\to\dot F_\tau. \]
Then
	\[  \tag{4} a_0=b_0\fo \forall z\in f\, \exists x\in\dot F_\sigma\, \exists y\in\dot F_\tau\, (z=\pair{x,y}), \]
	\[ \tag{5} a_{10}=b_{10}\fo \forall x\in\dot F_\sigma\, \exists y\in\dot F_\tau\, \exists z\in f\, (z=\pair{x,y}), \]
	\[  \tag{6} a_{11}=b_{11}\fo \forall z_0\in f\, \forall z_1\in f\, \forall x,y_0,y_1\, (z_0=\pair{x,y_0}\land z_1=\pair{x,y_1}\imp y_0=y_1). \]
	
We aim for 
\[   \mb e_1a=\mb e_1b\fo f\in \dot F_{\sigma\tau}. \]

As in the proof of Theorem \ref{choice}, it follows from (5) that $\lambda c.(a_{10}c)_0=_{\sigma\tau}\lambda d.(b_{10}d)_0$. Therefore
\[  \pair{\lambda c.(a_{10}c)_0,\lambda d.(b_{10}d)_0,  \tuep g{\sigma\tau}}\in \dot F_{\sigma\tau}, \]
where $g:=\lambda c.(a_{10}c)_0$. We thus want $\mb e_1$ such that 
\[   (\mb e_1a)_0\simeq \lambda c.(a_{10}c)_0, \]
\[   (\mb e_1a)_1=(\mb e_1b)_1\fo f= \tuep g{\sigma\tau}. \]

By definition and Lemma \ref{abs},
\[  \tuep g{\sigma\tau}=\{\pair{c,d,\vpair{ \tuep c{\sigma}, \tuep c{\tau}}}\colon c=_\sigma d\land (a_{10}c)_0=_\tau e\}. \]

($\subseteq$) Let $\pair{\tilde{c},\tilde d,z}\in f$. We aim for $((\mb e_1a)_1\tilde c)_0=((\mb e_1b)_1\tilde d)_0\fo z\in  \tuep g{\sigma\tau}$. By (4), $(a_0\tilde c)_0=_\sigma (b_0\tilde d)_0$ and
\[    (a_0\tilde c)_{11}=(b_0\tilde d)_{11}\fo z=\pair{ \tuep c{\sigma}, \tuep c{\tau}}, \]
where  $c\simeq (a_0\tilde c)_0$ and $e\simeq (a_0\tilde c)_{10}$. 
By Lemma \ref{pairs}, let $\mb z$ be a closed application term such that  for all $x,y,z\in\va$,
\[  \mb z\fo z=\pair{x,y}\imp z=\vpair{x,y}. \]
Then 
\[  \mb z(a_0\tilde c)_{11}=\mb z(b_0\tilde d)_{11}\fo z=\vpair{ \tuep c{\sigma}, \tuep c{\tau}}. \]
By using (5), (6) and absoluteness, one obtains $(a_{10}c)_0=_\tau e$. Let $\mb e_1$ satisfy
\begin{align*}
((\mb e_1a)_1\tilde c)_{00}& \simeq (a_0\tilde c)_0,\\
((\mb e_1a)_1\tilde c)_{01}& \simeq \mb z (a_0\tilde c)_{11}.
\end{align*}
Then $\mb e_1$ is as desired. 

($\supseteq$) Let $\pair{c,d,\vpair{ \tuep c{\sigma}, \tuep c{\tau}}}\in  \tuep g{\sigma\tau}$, with $e\simeq (a_{10}c)_0$.  We aim for $((\mb e_1a)_1 c)_1=((\mb e_1b)_1d)_1\fo \vpair{ \tuep c{\sigma}, \tuep c{\tau}}\in f$.

By unravelling (5), we obtain that for some $z\in\va$, 
\[   \pair{(a_{10}c)_{10},(b_{10}d)_{10}, z}\in f, \]
\[   (a_{10}c)_{11}=(b_{10}d)_{11}\fo z=\pair{ \tuep c{\sigma}, \tuep c{\tau}}. \]
Let $\mb e_1$ be such that 
\[   ((\mb e_1a)_1c)_1\simeq \mb p (a_{10}c)_{10} (\mb {i_s}(\mb z(a_{10}c)_{11})), \]
where $\mb z$ is as above.

By $\lambda$-abstraction, one can find $\mb e$ satisfying the above equations.
\end{proof}

\begin{theorem}\label{choice sound}
For all finite types $\sigma$ and $\tau$ there exists a closed application term $\mb c$ such that $\czf$ proves
\[  \mb c\fo  \forall x^\sigma\, \exists y^\tau\, \vp(x,y)\imp \exists f^{\sigma\tau}\, \forall x^\sigma\,  \vp(x,f(x)). \]
\end{theorem}
\begin{proof}
A proof is obtained by combining Theorem \ref{choice} and Theorem \ref{arrow}.
Let 
\[ \vartheta_0(z):=\text{\lq $z$ is the set of natural numbers\rq}, \]  
\[ \vartheta_{\sigma\tau}(z):=\exists x\, \exists y\, (\vartheta_\sigma(x)\land \vartheta_{\tau}(y)\land z=x\to y). \] 
We are claiming that for all finite types $\sigma$ and $\tau$ there exists a closed application term $\mb c_{\sigma\tau}$ such that $\czf$ proves
\[   \mb c_{\sigma\tau}\fo \forall z_\sigma\, \forall z_\tau\, (\vartheta_\sigma(z_\sigma)\land \vartheta_\tau(z_\tau)\imp \psi(z_\sigma,z_\tau)), \]
where $\psi(z_\sigma,z_\tau)$ is
\[    \forall x\in z_\sigma\, \exists y\in z_\tau\, \vp(x,y)\imp \exists f\colon z_\sigma\to z_\tau\, \forall x\in z_\sigma\, \vp(x,f(x)). \] 
Let $\mb e_0$ be such that $\mb e_0\fo \vartheta_0(\dot \omega)$. By using $\mb e_0$ and  Theorem \ref{arrow}, for every finite type $\sigma$, we can find $\mb e_\sigma$ such that $\mb e_\sigma\fo \vartheta_\sigma(\dot F_\sigma)$. As $\czf\vdash \vartheta_\sigma(z_0)\land \vartheta_\sigma(z_1)\imp z_0=z_1$, by soundness (Theorem \ref{czf sound})  there is a $\mb u_\sigma$ such that 
\[   \mb u_\sigma\fo  \vartheta_\sigma(z_0)\land \vartheta_\sigma(z_1)\imp z_0=z_1 \]
for all $z_0,z_1\in\va$. By soundness as well,  there are $\mb i_{\sigma\tau}$ and $\mb j_{\sigma\tau}$ such that 
\begin{align*}
\mb i_{\sigma\tau}&\fo \psi(\dot F_\sigma,\dot F_\tau)\land z_\sigma=\dot F_\sigma \imp \psi(z_\sigma,\dot F_\tau), \\
\mb j_{\sigma\tau}&\fo  \psi(z_\sigma,\dot F_\tau)\land z_\tau=\dot F_\tau\imp \psi(z_\sigma,z_\tau), 
\end{align*}
for all $z_\sigma,z_\tau\in \va$. Finally, with the aid of $\mb e_\sigma$, $\mb e_\tau$, $\mb u_\sigma$, $\mb u_\tau$, $\mb i_{\sigma\tau}$, $\mb j_{\sigma\tau}$, and of the closed application term $\mb e$ from Theorem \ref{choice}, one can construct $\mb c_{\sigma\tau}$ as desired. 
\end{proof}

\begin{corollary}\label{czf choice sound}
For every theorem $\vp$ of $\czf+\ac_{\ft}$, there is a closed application term $\mb e$ such that  $\czf$ proves $\mb e\fo \vp$.  In particular, $\czf+\ac_{\ft}$ is consistent relative to $\czf$.	
\end{corollary}
\begin{proof}
By Theorem \ref{czf sound} and Theorem \ref{choice sound}.	
\end{proof}
\begin{corollary}
$\czf+\ac_{\ft}$ is conservative over  $\czf$ with respect to $\Pi^0_2$ sentences. 
\end{corollary}
\begin{proof}  
Let $\vp(x,y)$ be a bounded formula with displayed free variables and suppose that 
\[ \forall x\in\omega\, \exists y\in\omega\, \vp(x,y)\footnote{Of course we mean that, e.g., $\forall z\, (\vartheta_0(z)\imp \forall x\in z\, \exists y\in z\, \vp(x,y)))$
is provable, where $\vartheta_0(z)$ is a formula defining $\omega$.} \]
is provable in $\czf$ plus  $\ac_{\ft}$.
By the corollary above, we can find a closed application term $\mb e$ such that
\[ \czf\vdash \mb e\fo \forall x\in\dot\omega\, \exists y\in\dot\omega\, \vp(x,y). \]
In particular, 
\[ \czf\vdash \forall n\in\omega\, \exists m\in\omega\, (e\bar n)_1\fo \vp(\dot n,\dot m). \]
It is a routine matter (cf.\ also \cite[Chapter 4, Theorem 2.6]{M84}) to show that  realizability equals truth for  bounded  arithmetic formulas, namely,
\[ \czf\vdash \forall n_1,\ldots, n_k\in\omega\, (\psi(n_1,\ldots,n_k)\biimp \exists a,b\in A\, (a=b\fo \psi(\dot n_1,\ldots,\dot n_k)), \]
for $\psi(x_1,\ldots,x_k)$ bounded with all the free variables shown.   
We can then conclude 
\[ \czf \vdash \forall x\in\omega\, \exists y\in\omega\, \vp(x,y). \] 
\end{proof}

\section{Soundness for $\izf$}\label{sec IZF}
The theory $\izf$  (Intuitionistic Zermelo-Fraenkel set theory) shares the logic and language of $\czf$. Its axioms are 

1. \textbf{Extensionality},  

2. \textbf{Pairing},  

3. \textbf{Union},  

4. \textbf{Infinity},  

5. \textbf{Set induction},

6. \textbf{Separation}: $\forall x\, \exists y\, \forall z\, (z\in y\biimp z\in x\land \vp(z))$, for all formulae $\vp$,

7. \textbf{Collection}: $\forall u\in x\, \exists v\, \vp(u,v)\imp \exists y\, \forall u\in x\, \exists v\in y\, \vp(u,v)$, for all formulae $\vp$,

8.	\textbf{Powerset}: $\forall x\, \exists y\, \forall z\, (\forall u\in z\, (u\in x)\imp z\in y)$.\\

Thus $\izf$ is a strengthening of $\czf$ with bounded separation replaced by full separation and subset collection replaced by powerset. Note that powerset implies subset collection and strong collection follows from separation and collection.

Note that in $\izf$, due to the presence of powerset, the construction of $\va$ can proceed by transfinite recursion along the ordinals (cf.\ \cite{M84}).

\begin{theorem}\label{izf choice sound}
For every theorem $\vp$ of $\izf+\ac_{\ft}$, there is a closed application term $\mb e$ such that $\izf$ proves $\mb e\fo \vp$.  In particular, $\izf+\ac_{\ft}$ is consistent relative to $\izf$.	
\end{theorem}
\begin{proof}
The soundness for theorems of intuitionistic first order logic with equality follows immediately from Theorem \ref{int sound}.   As for nonlogical axioms, in view of Corollary \ref{czf choice sound}, it is sufficient to deal with separation and powerset. 

The argument for separation is similar to the corresponding argument for bounded separation in the proof of Theorem \ref{czf sound}, employing full separation in the background theory.

It thus remains to address powerset.  
Write $z\subseteq x$ for $\forall u\in z\, (u\in x)$.  We look for $\mb e$ such that for all $x\in\va$ there is a $y\in\va$ such that
	\[  \mb e\fo z\subseteq x \imp z\in y, \]
	for all $z\in\va$. 
	
On account of powerset, in $\izf$, we can define sets $\va_\alpha$, with  $\alpha$ ordinal (i.e., a transitive set of transitive sets), such that $\va=\bigcup_\alpha\va_\alpha$ and $\va_\alpha=\bigcup_{\beta\in\alpha}\ps(A\times A\times \va_\beta)$. Note that in $\czf$ the $\va_\alpha$'s are just classes. 
	
Given $x\in\va_\alpha$, let  
	\[  y=\{\pair{a,b,z}\in A\times A\times \va_\alpha \mid a=b\fo z\subseteq x\}. \]
The set $y$ exists by separation. Set 
	\[   \mb e:=\lambda a.\mb p a\mb {i_r}.  \]
	
It is easy to check that $y$  and $\mb e$ are as desired, once established that if $z\in \va$ and
$a=b\fo z\subseteq x$ then $z\in\va_\alpha$. This is proved by set induction by showing that for all  $u,v\in\va$:
	\begin{itemize}
		\item if $a=b\fo u\in v$ and $v\in\va_\alpha$, then $u\in\va_\beta$ for some $\beta\in\alpha$; 
		\item if $a=b\fo u=v$ and $v\in\va_\alpha$, then $u\in \va_\alpha$.
	\end{itemize}
\end{proof}

As before, we obtain the following.

\begin{corollary}
$\izf+\ac_{\ft}$ is conservative over  $\izf$ with respect to $\Pi^0_2$ sentences. 	
\end{corollary}

\section{Conclusions}
We defined an extensional notion of realizability that validates  $\czf$  along with all finite type axiom of choice $\ac_{\ft}$ provably in $\czf$. We have shown that one can replace $\czf$ with $\izf$. Presumably, this holds true for many other intuitionistic set theories as well.

There is a sizable number of well-known  \emph{extra principles} $P$ that can  be added to the mix, in the sense that  $T+P$ proves $\mb e\fo P$, for some closed application term $\mb e$, where  $T$ is either $\czf$ or $\izf$. This applies to arbitrary pca's in the case of large set axioms such as $\sf REA$ (Regular Extension axiom) by adapting \cite[Theorem 6.2]{R06}.
In the case of choice principles, this also applies to arbitrary pca's for Countable Choice, $\dc$ (Dependent Choice), $\rdc$ (Relativized Dependent Choice), and  $\pax$ (Presentation Axiom) by adapting the techniques of \cite{DR19}. Specializing to the case of the first Kleene algebra, one obtains extensional realizability of  $\markov$ (Markov Principle) and forms of  $\ip$ (Independence of Premise) adapting results from \cite[Section 11]{M84}, \cite{M86} , \cite[Section 7]{R06}.

We  claim that  realizability combined  with truth and the appropriate pca modeled on \cite{R05,R08}  yields the closure under the choice rule for finite types, i.e.,
\[ \text{If }T\vdash \forall x^\sigma\, \exists y^\tau\, \vp(x,y), \text{ then } T\vdash \exists f^{\sigma\tau}\, \forall x^\sigma\, \vp(x,f(x)) \]
for  large swathes of intuitionistic set theories.

Church's thesis,  
\[ \tag{{\sf CT}}\forall f\colon\omega\to\omega\,  \exists e\in\omega\, \forall x\in\omega\, ( f(x)\simeq \{e\}(x)),\]
where $\{e\}(x)$ is Turing machine application, and the finite type axiom of choice are incompatible in extensional finite type arithmetic  \cite{T77} (cf.\ \cite[Chapter 5, Theorem 6.1]{B85}).\footnote{The elementary recursion-theoretic reason that prevents Church's thesis from being extensionally realizable is the usual one: there is no type $2$ extensional index in Kleene's first algebra, that is, there is no $e\in\omega$ such that,  for all $a,b\in\omega$, if  $\{a\}(n)=\{b\}(n)$ for every $n\in\omega$, i.e., $a=_1b$, then $\{e\}(a)=\{e\}(b)$.} A fortiori, they are incompatible  on the basis of $\czf$, and  thus of $\izf$. However, negative versions of Church's thesis can still obtain in a universe in which $\ac_{\ft}$ holds. The assertion that no function from $\omega$ to $\omega$ is incomputable is known as  weak Church's thesis \cite{T73}: 
\[ \tag{{\sf WCT}}\forall f\colon\omega\to\omega\,  \neg\neg \exists e\in\omega\, \forall x\in\omega\, ( f(x)\simeq \{e\}(x)). \]
Using Kleene's first algebra, one can easily verify that {\sf WCT} is extensionally realizable in $\czf$. Therefore, $\czf$ augmented with both $\ac_{\ft}$ and ${\sf WCT}$ is consistent relative to $\czf$, and similarly for $\izf$.

Continuity principles are a hallmark of Brouwer's intuitionism. They are compatible  with finite type arithmetic (see  \cite{B85,T73,T77,Oosten97}) and also with set theory (see \cite{B85,M84,R05,R05a}).  They are known, though,  to invite  conflict with $\ac_{\ft}$ (see \cite[ Theorem 9.6.11]{TvD88}). However, as in the case of {\sf CT}, negative versions of them are likely to be compatible with $\ac_{\ft}$ on the basis of $\czf$ and $\izf$. Similar to the case of ${\sf CT}$,  one would expect that the assertion that no function from $\mathbb R$ to $\mathbb R$  is discontinuous  can go together with $\ac_{\ft}$. One obvious tool that suggests itself  here is extensional generic realizability based on Kleene's second algebra. We shall not venture into  this here and add the verification of this claim to the task list.

We conclude with the following remark. It is currently unknown whether one can provide a realizability model for  choice principles based  on larger type structures. Say that $I$ is a base if for every $I$-indexed family $(X_i)_{i\in I}$ of inhabited sets $X_i$ there exists a function $f\colon I\to \bigcup_{i\in I}X_i$ such that $f(i)\in X_i$ for every $i\in I$. Let  $\mathcal C$-$\ac$ say that every set $I$ in the class $\mathcal C$ is a base. The question is whether one can realize $\mathcal C$-$\ac$, where  $\mathcal C$ is the smallest $\Pi\Sigma$-closed class,   or even the smallest $\Pi\Sigma W$-closed class, without assuming choice in the background theory.

\bibliographystyle{plain}

\end{document}